\documentclass[11pt]{article}

\usepackage[a4paper, total={125mm, 185mm}]{geometry}
\usepackage{geometry}
\usepackage[english]{babel}
\usepackage{newlfont}
\usepackage{fancyhdr}
\usepackage{graphicx}
\usepackage{amsmath}
\usepackage{amsfonts}
\usepackage{amsthm}
\usepackage{latexsym}
\usepackage[utf8]{inputenc}
\usepackage{enumerate}
\usepackage{stmaryrd}
\usepackage{tikz}

\theoremstyle{plain}

\newtheorem{thm}{Theorem}[section]

\newtheorem{cor}[thm]{Corollary}
\newtheorem{lem}[thm]{Lemma}
\newtheorem{prop}[thm]{Proposition}

\theoremstyle{definition}
\newtheorem{defn}{Definition}[section]
\newtheorem{oss}{Remark}[section]
\newtheorem{example}{Example}[section]

\title{Betti numbers of Brill-Noether varieties \\ on a general curve}
\author{Camilla Felisetti and Claudio Fontanari}
\date{}

\DeclareMathOperator{\rank}{rank}

\newcommand{\ita}{\textit}

\begin{document}

\maketitle
\begin{abstract}
We compute the rational cohomology groups of the smooth Brill-Noether varieties $G^r_d(C)$, parametrizing linear series of degree $d$ and dimension exactly $r$ on a general curve $C$. As an application, we determine the whole intersection 
cohomology of the singular Brill-Noether loci $W^r_d(C)$, parametrizing complete 
linear series on $C$ of degree $d$ and dimension at least $r$.\\
\end{abstract}
\\

\section{Introduction}

The algebro-geometric framework of Brill-Noether theory is at least one century old (the classical reference is \cite[Anhang G]{S21}, going back to 1921), but completely rigorous proofs were elaborated only in the late seventies and then presented in the definitive account \cite{ACGH85}. 

To make a long story short, let $C$ be a smooth projective curve. 
The Brill-Noether locus $W^r_d(C)$ parametrizes complete linear series on $C$ of 
degree $d$ and dimension at least $r$:
$$ W^r_d(C)=\left \lbrace L\in Pic^d(C): \dim \vert L\vert \geq r \right\rbrace. $$
In the same setting, one defines $G^r_d(C)$ to be the variety parametrizing linear series of degree $d$ and dimension exactly $r$ on $C$:
$$ G^r_d(C):=\left \lbrace g^{r}_d\text{'s on }C \right\rbrace. $$

If $C$ is a general curve of genus $g$, then $W^r_d(C)$ has the expected dimension 
encoded by the \textit{Brill-Noether number}
$$\rho(g,d,r):= g-(r+1)(g-d+r),$$
as first proven in modern terms in \cite{GH80}.
Moreover, $W^r_d(C)$ is a Cohen-Macaulay variety, having 
$W^{r+1}_d(C)$ as its singular locus, as established in \cite{G82}. 
In the same paper it was proven that when $\rho>0$ then $G^r_d(C)$ is smooth 
(see also \cite[(3.7)]{AC81}).
Finally, in \cite{FL81} Fulton and Lazarsfeld showed that
both $G^r_d(C)$ and $W^r_d(C)$ are connected for $\rho\geq 1$. 
Indeed, the proof relies on the description of $W^r_d(C)$ as a \ita{degeneracy locus}, which we are going to recall in Section \ref{deg}. 

Of course, connectedness of a variety $X$ can be rephrased in terms of cohomology as follows: $H^0(X) = \mathbb{Q}$. 
As pointed out in \cite[Remark 1.9]{FL81}, one expects that 
such a connectedness theorem could be extended to a Lefschetz-type result on higher 
cohomology. For Brill-Noether loci $W^r_d(C)$ this strategy has been implemented 
by Debarre in \cite{D00}, yielding a bijection $H^p(Pic^d(C)) \to H^p(W^r_d(C))$ 
in a limited range of integers $p$ (see \cite[Example (2.6)]{D00}). Debarre's 
results have been later adapted also to rational homotopy groups (see \cite[Corollary 1.7]{MP03}).

Here instead we focus on the Brill-Noether varieties $G^r_d(C)$ and at the end 
of Section \ref{bn} we are able to compute in all degrees its cohomology groups (see Theorem \ref{coh}) by adding as a crucial ingredient the computation of the Euler characteristic of $G^r_d(C)$ due to Parusi\'{n}ski and Pragacz in \cite{PP95}. 

Next we turn to the singular varieties $W^r_d(C)$. Recall that, when working with non-smooth varieties, the natural topological invariant to consider is \ita{intersection cohomology}, introduced by Goresky and MacPherson to restore all Hodge theoretic properties, such as Poincaré duality or Lefschetz theorems, failing for ordinary cohomology of singular spaces.
 
As an application of Theorem \ref{coh}, in Section \ref{ic} we determine the whole intersection 
cohomology of the Brill-Noether loci $W^r_d(C)$ (see Corollary \ref{icoh}).

We work over the complex field $\mathbb{C}$ and we consider cohomology with rational coefficients.

This research project was partially supported by GNSAGA of INdAM and by PRIN 2017 “Moduli Theory and Birational Classification”.

\section{Degeneracy loci}\label{deg}	

Let $X$ be a smooth projective variety and let $E,F$ be vector bundles of ranks respectively $e$ and $f$. Let $\gamma:E\rightarrow F$ be a morphism of vector bundles. 
Consider the $s$-th degeneracy locus of $\gamma$
$$ D_s(\gamma):=\left\lbrace x\in X\mid \rank \gamma_x\leq s\right\rbrace$$
endowed with its reduced scheme structure. Fulton and Lazarsfeld prove the following Bertini type result.

\begin{thm}[\cite{FL81}, Theorem 1.1]\label{fula}
Let $X$ be a smooth projective variety and let $\gamma:E\rightarrow F$ be a morphism of vector bundles on $X$ of ranks $e$ and $f$. If $Hom(E,F)$ is ample then 
\begin{enumerate}[(i)]
    \item $D_s(\gamma)$ is non-empty if $\delta(s)\geq 0$;
    \item $D_s(\gamma)$ is connected if $\delta(s)> 0$
\end{enumerate} where $\delta(s)= \dim X - (e-s)(f-s)$ is the expected dimension of $D_s(\gamma)$. 
\end{thm} 

The proof is based on the following construction. Suppose $e\leq f$ and let $G:=\mathrm{Gr}(e-s,E)\xrightarrow{\pi} X$ be a fiber bundle with fibre $G_x:=\mathrm{Gr}(e-s,E_x)\cong \mathrm{Gr}(e-s,e)$, i.e. any element $y\in G$ can be represented by a vector subspace $U_y\subset E_{\pi(y)}$ of dimension $e-s$.

Denote by $S$ the tautological bundle on $G$: observe that since  $S_y\cong U_y$ for any point $y\in G$, $S$ is a subbundle of $\pi^*E$.
We set $\tau$ to be the composition
$\tau:S\hookrightarrow \pi^*E\xrightarrow{\pi^*\gamma}\pi^*F,$
and define 
$$Y:=\left\lbrace y\in G\mid \tau(y)=0\right\rbrace,$$
which is called the \ita{canonical blow-up of $D_s(\gamma)$}.
Observe that the support of $Y$ is simply the set of all $y\in G$ such that $U_y\subset \ker \gamma_{\pi(y)}$. This implies in particular that for all $y\in Y$ the rank of $\gamma_{\pi(y)}$ is at most $s$, i.e. $\pi(y)\in D_s(\gamma)$.

The following facts, stated for instance in \cite[Chapter II, \S 2]{ACGH85} and \cite[(1.1)]{D00}, show that $Y$ provides a resolution on singularities of the degeneracy locus. 
\begin{lem}\label{descriptionY}
Under the assumption of Theorem \ref{fula},
let $\phi:Y\rightarrow D_s(\gamma)$ be the restriction of $\pi:G\rightarrow X$ to $Y$.
\begin{enumerate}[(i)]
    \item $Y$ is nonsingular variety of dimension $\delta(s)$;
    \item $\phi$ is a resolution of singularities of $D_s(\gamma)$;
    \item the fibre of $\phi$ over $D_l(\gamma)\setminus D_{l-1}(\gamma)$ is isomorphic to $\mathrm{Gr}(e-s,e-l)$.
\end{enumerate}
\end{lem}

The cohomology of $Y$ can be understood in terms of the cohomology of the Grassmann bundle $G$, which is computed for example in \cite[Proposition 14.6.5]{F84}. 
Indeed, we have a Lefschetz-type result.
\begin{prop}\label{cohomy=cohomg}
Under the assumption of Theorem \ref{fula},
let $\iota:Y\rightarrow G$ be the inclusion of $Y$ inside $G$. Then the map 
$$\iota^*:H^i(G)\rightarrow H^i(Y)$$
is an isomorphism for $i<\delta(s)$ and injective for $i=\delta(s).$
\end{prop}
\begin{proof}
By \cite[Proposition 1.2]{FL81} we have that $H^j(G\setminus Y)=0$ for $j\geq \dim X +(f+s)(e-s)$, so Lefschetz duality implies $H^{2\dim G-j}(G,Y)=H^j(G\setminus Y)=0$.

Setting $i=2\dim G-j$ and noticing that $\dim G=\dim X+s(e-s)$, we see that $j\geq \dim X +(f+s)(e-s)$ if and only if $i\leq \delta(s)$. 
As a result 
$$H^i(G,Y)=0 \quad \forall i\leq \delta(s).$$
The statement now follows applying the long exact sequence for the relative cohomology. 
\end{proof}

\section{Brill-Noether theory}\label{bn}

In this section we recall for ease of the reader the classical description of the  Brill-Noether loci $W^r_d(C)$ for a smooth curve $C$ as degeneracy loci $D_s(\gamma)$ for a suitable morphism $\gamma$ between vector bundles on $Pic^d(C)$. 
For further details we refer to \cite[Chapter IV, \S 3]{ACGH85}.
 
Let $C$ be a smooth projective curve and denote by $\mathcal{L}$ a fixed Poincaré line bundle on $Pic^d(C)\times C$. Let $D$ be a fixed divisor of degree 
$$deg(D)=f\geq 2g-d-1$$
on $C$ and let $\Gamma=Pic^d(C)\times D$ be the corresponding the product divisor on $Pic^d(C)\times C$.
Denoting by $\nu:Pic^d(C)\times C\rightarrow Pic^d(C)$ the projection onto the first factor, we have that 
$\nu_*\mathcal{L}(\Gamma)$ is a vector bundle on $Pic^d(C)$ of rank 
$$e= d+f-g+1.$$
The higher direct image sequence of 
$$ 0\rightarrow \mathcal{L}\rightarrow \mathcal{L}(\Gamma)\rightarrow \mathcal{L}(\Gamma)/\mathcal{L}\rightarrow 0$$
yields a morphism of vector bundles
$$\gamma: \nu_*\mathcal{L}(\Gamma)\rightarrow \nu_*(\mathcal{L}(\Gamma)/\mathcal{L}).$$
Setting
\begin{align*}
X&=Pic^d(C),\\
E&=\nu_*\mathcal{L}(\Gamma),\\
F&=\nu_*(\mathcal{L}(\Gamma)/\mathcal{L}),  \\    
s&=f+d-g-r,
\end{align*}
one has that $Hom(E,F)$ is ample and 
\begin{equation}\label{brilldeg}
W^r_d(C)=D_s(\gamma).
\end{equation}
A simple computation shows that 
$$\rho(g,r,d)=\delta(s),$$
i.e. the expected dimension of $W^r_d(C)$ coincides with the expected dimension of the corresponding degeneracy locus.

By applying the construction in Section \ref{deg}, we can define the Grassmann bundle $\mathcal{G}$ on $Pic^d(C)$ given by the set of $(r+1)$-planes in $E$.
In this setting, it is easy to see that variety $Y$ is exactly the space $G^r_d(C)$ parametrizing $r-$dimensional linear series on $Pic^d(C)$. By Lemma \ref{descriptionY} we have that  $\phi:G^r_d(C)\rightarrow W^r_d(C)$ is a resolution of singularities. 

Suppose now that $C$ is a general curve, so that $W^r_d(C)$ and $G^r_d(C)$ have the expected dimension $\rho(g,r,d)$. The following theorem computes the whole cohomology of $G^r_d(C)$.


\begin{thm}\label{coh}
Let $C$ be a general curve of genus $g$. Then there is a linear isomorphism
$$
H^i(G^r_d(C))\cong
\begin{cases} 
H^i(\mathcal{G}) &\text{ if }i< \rho\\
\mathbb{Q}^k, \ k=(-1)^{\rho}(\Phi- 2\sum_{j<\rho}
(-1)^j b_j(\mathcal{G})) &\text{ if }i=\rho\\
H^{2\rho-i}(\mathcal{G}) &\text{ if }i>\rho\\
\end{cases}
$$
where $\rho = \rho(g,r,d)$, $b_j$ is the $j$-th Betti number, and  $\Phi = \Phi(g,d,r)$ denotes as in \cite[p.815]{PP95}, the right hand side of the formula of \cite[Lemma 3.2]{PP95}. 
\end{thm}
\begin{proof}
Observe that by Proposition \ref{cohomy=cohomg}
$$H^i(G^r_d(C))\cong H^i(\mathcal{G})\quad \forall i<\rho(g,r,d).$$
Since $G^r_d(C)$ is smooth of dimension $\rho(g,r,d)$, the result for $i>\rho(g,r,d)$ follows from Poincaré duality. 
Finally, the dimension of the middle cohomology group $H^{\rho(g,r,d)}(G^r_d(C))$
follows from the computation of the Euler characteristic in \cite[Theorem 3.8]{PP95}:
$$
\chi(G^r_d(C)) =  \sum_i (-1)^i\dim H^i(G^r_d(C)) = \Phi(g,d,r).
$$
\end{proof}

\begin{example}
To give an example of how our formula works, we consider for instance a smooth curve of genus $4$. We compute the cohomology of $G^0_3(C)$, which is a smooth variety of dimension $3$. We can repeat the construction in Section 2 and write $W^0_3(C)=D_3(\gamma)$, where $\gamma:E\rightarrow F$ is a linear map between two vector bundles of rank $4$ on $Pic^3(C)$. In this case $\mathcal{G}$ is isomorphic to the total space of $\mathbb{P}(E)$, so its Betti numbers can be computed via the K\"unneth formula by the projective bundle theorem.  Hence we have a linear isomorphism
$$ H^i(G^0_3(C))\cong H^i(\mathcal{G})=\begin{cases}
\mathbb{Q} &i=0,6\\
\mathbb{Q}^{8}&i=1,5\\
\mathbb{Q}^{29}&i=2,4\\
\end{cases} .$$
Following \cite[Example 3.6]{PP95} one can show that $\Phi(4,0,3)=-20$ and applying Theorem \ref{coh} we get 
$$\dim H^3(G^0_3(C))=64,$$
thus determining all cohomology groups of $G^0_3(C)$.
\end{example}

\section{Intersection cohomology}\label{ic}

Let us recall some preliminaries. For further details we refer to \cite[\S 6.2]{GM83}. 
\begin{defn}
Let $\phi:Y\rightarrow D$ be a proper map of algebraic varieties. For all $k\geq0$ let 
$$D_k=\left\lbrace x\in D\mid \dim \phi^{-1}(x)\geq k\right\rbrace.$$
We say that $\phi$ is \ita{small} if 
\begin{equation}\label{small}
\dim Y\leq \dim D_k+2k
\end{equation}
for all $k$ and equality holds only for $k=0$.
\end{defn}

\begin{prop}
Let $\phi:Y\rightarrow D$ be a small map of algebraic varieties. Then
$$IH^i(Y)=IH^i(D) \quad \forall i\geq 0.$$
In particular, if $Y$ is smooth then we have $H^i(Y)=IH^i(Y)=IH^i(D)$.
\end{prop}

We are now going to prove that the map $\phi:Y\rightarrow D_s(\gamma)$ 
introduced in Lemma \ref{descriptionY} is a small resolution of singularities. This will give us the intersection cohomology of $D_s(\gamma)$.

\begin{prop}\label{smallres}
Let $\phi:Y\rightarrow D_s(\gamma)$ be the restriction of the Grassmann bundle projection $\pi:G\rightarrow X$. Suppose that $D_l(\gamma)\setminus D_{l-1}(\gamma)$ has the expected dimension for each $l\leq s$. Then $\phi$ is a small resolution of singularities.
\end{prop}
\begin{proof}
By Lemma \ref{descriptionY}, we now that $Y$ is a resolution to singularities. Consider the stratification
$$D_0(\gamma)\subset D_1(\gamma)\subset\ldots \subset D_s(\gamma). $$
By assumption, $D_l(\gamma)\setminus D_{l-1}(\gamma)$ is a smooth variety of dimension $\delta(l)=\dim X-(e-l)(f-l)$. Moreover, by item $(ii)$ of Lemma \ref{descriptionY},
$$D_{s-l}(\gamma)=\left\lbrace x\in D_s(\gamma)\mid \dim \phi^{-1}(x)\geq l(e-s)\right\rbrace.$$
A simple computation shows that \eqref{small} is satisfied and thus $\phi$ is small.
\end{proof}

\begin{cor}\label{corih}
In the notation of Section \ref{deg} we have
$$IH^i(D_s(\gamma))\cong H^i(Y)$$
\end{cor}

Finally, we are able to compute the intersection cohomology of all Brill-Noether loci.

\begin{cor} \label{icoh}
Let $C$ be a general curve of genus $g$. Then there is a linear isomorphism
$$IH^i(W^r_d(C))\cong \begin{cases} 
H^i(\mathcal{G}) &\text{ if }i< \rho\\
\mathbb{Q}^k, \ k=(-1)^{\rho}(\Phi- 2\sum_{j<\rho}
(-1)^j b_j(\mathcal{G})) &\text{ if }i=\rho\\
H^{2\rho-i}(\mathcal{G}) &\text{ if }i>\rho\\
\end{cases}$$
where $\rho = \rho(g,r,d)$ and $\Phi = \Phi(g,d,r)$.
\end{cor}
\begin{proof}
Consider the stratification 
$$ W^r_d(C)\supset W^{r+1}_d(C)\supset \ldots$$
Under the identification \eqref{brilldeg} we have
$W^{r+k}_d(C)=D_{s-k}(\gamma)$ and by \cite[Lemma 2.9 and observation below Theorem 3.1]{PP95} the complement $W^{r+k}_d(C)\setminus W^{r+k+1}_d(C)$ has the expected dimension. The result now follows by Proposition \ref{smallres} and Theorem \ref{coh}.
\end{proof}

Notice that this result is consistent with the computations in \cite[3.7 and 3.8]{PP95}, which imply $\chi_{IH}(W^r_d(C)) = \chi(G^r_d(C))$. 

\begin{oss}
We point out that, in the cohomological degrees $i$ in which Debarre's isomorphism $H^i(W^r_d(C))\cong Pic^d(C)$ holds, the natural map $H^i(W^r_d(C))\rightarrow IH^i(W^r_d(C))$ is an injection and it coincides with the map $H^*(Pic^d(C))\rightarrow H^*(\mathcal{G})$ given by the Leray-Hirsch theorem. Indeed, Debarre's map is an isomorphism of mixed Hodge structures and the cohomology groups $H^i(W^r_d(C))$ carry a pure Hodge structure. Since the kernel of $H^i(W^r_d(C))\rightarrow IH^i(W^r_d(C))$ is contained in the $(i-1)$-th piece $W_{i-1}$ of the weight filtration on $H^i(W^r_d(C))$ (see for instance \cite[Corollary 5.43]{PS08}), we deduce that in this case ordinary cohomology can be identified with a proper subset of intersection cohomology.
\end{oss}

\bibliographystyle{siam}
\bibliography{main}

\newpage

\noindent
Camilla Felisetti \newline
Dipartimento di Matematica \newline
Universit\`a di Trento \newline
Via Sommarive 14 \newline
38123 Trento, Italy. \newline
E-mail: camilla.felisetti@unitn.it

\vspace{0.7cm}

\noindent
Claudio Fontanari \newline
Dipartimento di Matematica \newline
Universit\`a di Trento \newline
Via Sommarive 14 \newline
38123 Trento, Italy. \newline
E-mail: claudio.fontanari@unitn.it

\end{document}